\documentclass{amsart}
\usepackage[utf8]{inputenc}
\usepackage{latexsym,amssymb,amsmath,amsthm,amscd,graphicx}
\usepackage{amsfonts}
\usepackage{xcolor}
\usepackage{tikz}
\usepackage[english]{babel}
\usepackage[a4paper,bindingoffset=0.2in,%
            left=1in,right=1in,top=1in,bottom=1in,%
            footskip=.25in]{geometry}
\newtheorem{theorem}{Theorem}[section]
\newtheorem{prop}[theorem]{Proposition}
\newtheorem{lemma}[theorem]{Lemma}
\newtheorem{remark}[theorem]{Remark}
\newtheorem{cor}[theorem]{Corollary}
\newtheorem{example}[theorem]{Example}
\numberwithin{equation}{section}

\begin{document}

\title[Birkhoff angles in normed spaces]{On Birkhoff angles in normed spaces}

\author[H.~Gunawan]{Hendra Gunawan}
\address{Faculty of Mathematics and Natural Sciences, Bandung Institute of Technology, Bandung 40132, Indonesia}
\email{hgunawan@math.itb.ac.id}

\author[M.~Jamaludin]{Muhamad Jamaludin}
\address{Faculty of Mathematics and Natural Sciences, Bandung Institute of Technology, Bandung 40132, Indonesia}
\email{mjamaludin@students.itb.ac.id}

\author[M.D.~Pratamadirdja]{Mas Daffa Pratamadirdja}
\address{Faculty of Mathematics and Natural Sciences, Bandung Institute of Technology, Bandung 40132, Indonesia}
\email{pratamadirdja@students.itb.ac.id}



\begin{abstract}
    Associated to Birkhoff orthogonality, we study Birkhoff angles in a normed space and present some of their
    basic properties. We also discuss how to decide whether an angle is more acute or more obtuse than another.
    In addition, given two vectors $x$ and $y$ in a normed space, we study the formula for Birkhoff `cosine' of
    the angle from $x$ to $y$ from which we can, in principal, compute the angle. Some examples will be presented.
\end{abstract}

\maketitle

\section{Introduction}

Let $(X,\|\cdot\|)$ be a real normed space. (Throughout this note, $X$ will be assumed so, unless otherwise stated.)
Unlike in an inner product space, there is not a standard definition of orthogonality, let alone the definition of
angles, in $X$. We know, for instance, the notion of Pythagorean orthogonality (denoted by $\perp_P$), which states that
\[
x\perp_P y ~{\rm if~and~only~if}~ \|x-y\|^2=\|x\|^2+\|y\|^2,
\]
and the notion of isosceles orthogonality (denoted by $\perp_I$), which states that
\[
x\perp_I y ~{\rm if~and~only~if}~ \|x+y\|=\|x-y\|,
\]
for $x,y\in X$, as introduced by R.C.~James in \cite{James1945}. Both of these definition coincides with the usual definition of
orthogonality when $X$ is an inner product space and the norm is induced by the inner product.

In 1935, G. Birkhoff \cite{Birkhoff} introduced a different notion of orthogonality in a normed space, inspired by the property
of the tangent to a circle in the Euclidean space (see Figure 1.1). A vector $x\in X$ is said to be $B$-orthogonal to another
vector $y\in X$, denoted by $x \perp_B y$, if and only if
\begin{equation}\label{B-ineq}
\|x+\lambda y\|\ge \|x\|
\end{equation}
for every $\lambda \in \mathbb{R}$. This notion of orthogonality is then known as Birkhoff orthogonality, as studied in \cite{James}.
One may observe that this definition also coincides with the usual orthogonality when $X$ is an inner product space and
the norm is induced by the inner product. However, unlike Pythagorean orthogonality and isosceles orthogonality, Birkhoff
orthogonality is not symmetric: $x\perp_B y$ does not imply $y \perp_B x$. For discussions on various notions of orthogonality
in normed spaces, see \cite{AMW, BHMT, Par}.

\begin{center}
\begin{tikzpicture}
               \draw [thick, gray, ->] (0,-3) -- (0,3)  
               node [above, black] {$x_2$};              

               \draw [thick, gray, ->] (-3,0) -- (3,0)      
               node [right, black] {$x_1$};              
               \draw [green,line width=1pt] (0,0) circle (1cm) ;
               \draw [thick, black, ->] (0,0) -- (-0.7071,0.7071)  ;    
               \draw [thick, black, ->] (0,0) -- (0.7071,0.7071)  ;    
               \draw [thick, gray, -] (1.4142+1,-1) -- (-1,1.4142+1)  ;    
               \draw [thick, red, ->] (0,0) -- (-0.5,1.4142+0.5)  ;    
               \draw [thick, red, ->] (0,0) -- (1.4142+0.5,-0.5)  ;    
               \draw [thick, red, ->] (0,0) -- (0,1.4142)  ;    
               \draw [thick, red, ->] (0,0) -- (1.4142,0)  ;    
               \draw [thick, red, ->] (0,0) -- (0.5,1.4142-0.5)  ;    
               \draw [thick, red, ->] (0,0) -- (1.4142-0.5,0.5)  ;    
               \node [right] at (0.7071,0.7071) {$x$};
               \node [left] at (-0.7071,0.7071) {$y$};
               \node [right,red] at (0,1.4142) {$x+\lambda y$};
\end{tikzpicture}
\end{center}
\begin{center}
\small{
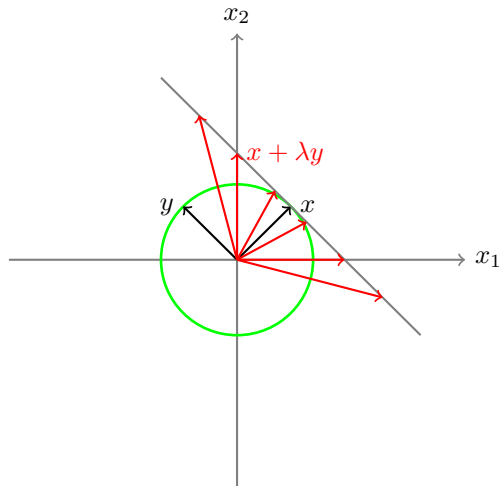
{\tt Figure 1.1}. A vector $x$ is $B$-orthogonal to another vector $y$ in $\ell_2^2$ ($\mathbb{R}^2$ with the usual
norm). The circle has radius $\|x\|$.
}
\end{center}

\begin{center}
\begin{tikzpicture}
               \draw [thick, gray, ->] (0,-3) -- (0,3)  
               node [above, black] {$x_2$};              

               \draw [thick, gray, ->] (-3,0) -- (3,0)      
               node [right, black] {$x_1$};              
               \draw [thick, green, -] (1,1)--(1,-1)--(-1,-1)--(-1,1)--(1,1)  ;    
               \draw [thick, black, ->] (0,0)--(1,0.5)  ;    
               \draw [thick, black, ->] (0,0)--(0,1)  ;    
               \draw [thick, red, ->] (0,0)--(1,0)  ;    
               \draw [thick, red, ->] (0,0)--(1,-0.5)  ;    
               \draw [thick, red, ->] (0,0)--(1,-1)  ;    
               \draw [thick, red, ->] (0,0)--(1,1)  ;    
               \draw [thick, red, ->] (0,0)--(1,-1.5)  ;    
               \draw [thick, red, ->] (0,0)--(1,1.5)  ;    
               \draw [thick, gray, -] (1,-2.5)--(1,2.5)  ;    
               \draw [thick, blue, -] (0,-2.5)--(0,2.5)  ;    
               \node [right] at (1,0.5) {$x$};
               \node [left] at (0,1) {$y$};
               \node [right,red] at (1,1) {$x+\lambda y$};
\end{tikzpicture}
\end{center}
\begin{center}\small{
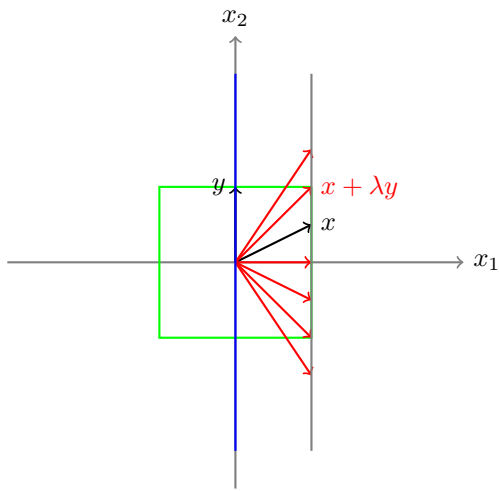
{\tt Figure 1.2}. A vector $x$ is $B$-orthogonal to another vector $y$ in $\ell_2^\infty$ ($\mathbb{R}^2$
with the supremum norm). The square is the set of vectors with length $\|x\|$. Note that here the vector $y$
is not $B$-orthogonal to $x$.
}
\end{center}

\medskip

Once we define an orthogonality criterion in $X$, we may ask how we can define angles in $X$ that are compatible with the
orthogonality defined there. If we use Pythagorean orthogonality, then we can define the angle {\it between} two nonzero
vectors $x$ and $y$ in $X$ by
\[
A_P(x,y):=\arccos \frac{\|x\|^2+\|y\|^2-\|x-y\|^2}{2\|x\|\|y\|}.
\]
If we adopt isosceles orthogonality, then we can define the angle between two nonzero vectors $x$ and $y$ in $X$ by
\[
A_I(x,y):=\arccos \frac{\|x+y\|^2-\|x-y\|^2}{4\|x\|\|y\|}.
\]
(See \cite{GLN} for some properties of Pythagorean angles and isosceles angles.) Now, if we opt to define Birkhoff orthogonality
in $X$, how should we define the angles {\it from} a vector {\it to} another vector in $X$? To answer this question, we
must dig deeper to the case where $X$ is an inner product space. This paper presents some results that might be of interest
to the readers, {\bf especially undergraduate students who are interested in geometry of normed spaces}. For further readings,
we refer the readers to \cite{BHMT, Milicic, ZWLY}.

\section{Preliminary Observation: Acute and Obtuse B-Angles}

For the moment, let $(X,\langle\cdot,\cdot\rangle)$ be an inner product space, and $\|x\|:=\langle x,x\rangle^{1/2}$ be the induced
norm on $X$. Let $x\in  X$ and $y\in X\setminus\{0\}$. By simple calculations, one may observe that
\begin{itemize}
    \item when $\langle x,y\rangle\geq 0$ (that is, $x$ and $y$ form an acute angle), the inequality (\ref{B-ineq}) holds
    for $\lambda \leq -\frac{2\langle x,y\rangle}{\|y\|^2}$ or $\lambda \ge 0$.
    \item when $\langle x,y\rangle \leq 0$ (that is, $x$ and $y$ form an obstuse angle), the inequality (\ref{B-ineq}) holds
    for $\lambda \leq 0$ or $\lambda \geq -\frac{2\langle x,y\rangle}{\|y\|^2}$.
    \item when $\langle x,y\rangle =0$ (that is, $x$ and $y$ are orthogonal to each other), the inequality (\ref{B-ineq}) holds
    for every $\lambda \in \mathbb{R}$.
\end{itemize}
The sets of the values of $\lambda$ that satisfies the inequality (\ref{B-ineq}) are visualized in the following figure:

\begin{center}
   \begin{tikzpicture}
               \draw [thick, black, -] (-3,0) -- (3,0) ;
               \draw [thick, blue, ->] (-1,0.25)--(-3,0.25) ;
               \draw [thick, blue, ->] (0,0.25)--(3,0.25) ;
               \draw [thick, black, -] (-1,-0.25) -- (-1,0.25) ;
               \draw [thick, black, -] (0,-0.25) -- (0,0.25) ;
               \fill (0,0) circle (0.05);
               \fill (-1,0) circle (0.05);
               \node [below] at (-1,-0.25) {$-\gamma$};
               \node [below] at (0,-0.25) {$0$};
\end{tikzpicture}
\\
\begin{tikzpicture}
               \draw [thick, black, -] (-3,0) -- (3,0) ;
               \draw [thick, blue, ->] (0,0.25)--(-3,0.25) ;
               \draw [thick, blue, ->] (1,0.25)--(3,0.25) ;
               \draw [thick, black, -] (0,-0.25) -- (0,0.25) ;
               \draw [thick, black, -] (1,-0.25) -- (1,0.25) ;
               \fill (1,0) circle (0.05);
               \fill (0,0) circle (0.05);
               \node [below] at (0,-0.25) {$0$};
               \node [below] at (1,-0.25) {$\gamma$};
\end{tikzpicture}
\\
\begin{tikzpicture}
               \draw [thick, black, -] (-3,0) -- (3,0) ;
               \draw [thick, blue, <->] (3,0.25)--(-3,0.25) ;
               \draw [thick, black, -] (0,-0.25) -- (0,0.25) ;
               \fill (0,0) circle (0.05);
               \node [below] at (0,-0.25) {$0$};
\end{tikzpicture}
\end{center}
\begin{center}\small{
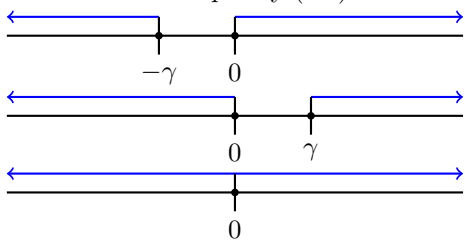
{\tt Figure 2.1}. The sets of the values of $\lambda$ for which the inequality (\ref{B-ineq}) holds when
(a) $\langle x,y\rangle\geq0$, (b) $\langle x,y\rangle\leq0$, and (c) $\langle x,y\rangle=0$, where
$\gamma:=\frac{2|\langle x,y\rangle|}{\|y\|^2}$.
}
\end{center}

\medskip

Note that when $\langle x,y\rangle\not=0$ (that is, $x$ and $y$ form a proper acute or obtuse angle), the value of
$\gamma:=\frac{2|\langle x,y\rangle|}{\|y\|^2}$ is stricly positive. This means that, there exist some values of
$\lambda$ for which the inequality (\ref{B-ineq}) fails to hold.

We now go back to normed spaces in general. Based on the above observation, we can define acute and obstuse angle
in a normed space $(X,\|\cdot\|)$ through the sets of the values of $\lambda$ satisfying the inequality (\ref{B-ineq}).
Let $x,y\in X$. We say that
\begin{itemize}
    \item $x$ forms an \textbf{acute B-angle} to $y$, denoted by $x\,A_B\,y$, if the inequality (\ref{B-ineq}) holds for every $\lambda\geq 0$.
    \item $x$ forms an \textbf{obtuse B-angle} to $y$, denoted by $x\,O_B\,y$, if the inequality (\ref{B-ineq}) holds for every $\lambda\leq 0$.
    \item $x$ is \textbf{B-orthogonal} to $y$, denoted by $x\perp_B y$, if $x$ forms an acute B-angle and an obstuse B-angle to $y$
    simultaneously.
\end{itemize}

As studied in \cite{MJ, MDP}, we have the following propositions. We leave the proof of the first proposition to the readers.

\begin{prop}
Let $x,y\in X\setminus \{0\}$ and $a,b\in\mathbb{R}\setminus\{0\}$.
{\parindent=0cm
\begin{enumerate}
    \item Suppose that $x\,A_B\,y$. If $ab>0$ (that is, $a$ and $b$ have the same sign), then $ax\,A_B\,by$.
    If $ab<0$, then $ax\,O_B\,by$.
    \item Suppose that $x\,O_B\,y$. If $ab>0$, then $ax\,O_B\,by$. If $ab<0$, then $ax\,A_B\,by$.
\end{enumerate}
}
\end{prop}

\begin{prop}
Let $x,y\in X$. Then we have
{\parindent=0cm
\begin{enumerate}
    \item $x\,A_B\,y$ if and only if there exists $\delta>0$ such that the inequality (\ref{B-ineq}) holds for every $\lambda\in [0,\delta)$.
    \item $x\,O_B\,y$ if and only if there exists $\delta>0$ such that the inequality (\ref{B-ineq}) holds for every $\lambda\in (-\delta,0]$.
    \item $x\perp_B y$ if and only if there exists $\delta>0$ such that the inequality (\ref{B-ineq}) holds for every $\lambda\in (-\delta,\delta)$
    simultaneously.
\end{enumerate}
}
\end{prop}

\begin{proof}
We shall only prove the first statement, as the second one can be proven in a similar way and
the third one is a consequence of the first and the second ones. Now, the `only if' part is immediate,
and so we only need to prove the `if' part. Suppose that there exists $\delta>0$ such that the
inequality $\|x+\lambda y\|\ge \|x\|$ holds for every $\lambda\in [0,\delta)$. If $x$ or $y$ equals $0$, then
the inequality obviously holds for every $\lambda\ge 0$. So assume that $x,y\not=0$. Suppose that, to the contrary,
there exists $\lambda^\prime\ge \delta$ such that $\|x+\lambda^\prime y\|<\|x\|$. Choose $n\in\mathbb{N}$
such that $\frac{\lambda^\prime}{n}\in(0,\delta)$. By the hypothesis,
we have $\left\|x+\frac{\lambda^\prime}{n} y\right\|\ge \|x\|$. But, by the triangle inequality, we find that
\begin{align*}
\left\|x+\frac{\lambda^\prime}{n}y\right\|&=\left\|\frac{n-1}{n}x+\frac{1}{n}(x+\lambda^\prime y)\right\|\\
&\le \frac{n-1}{n}\|x\|+\frac{1}{n}\|x+\lambda^\prime y\|\\
&< \frac{n-1}{n}\|x\|+\frac{1}{n}\|x\|\\
&=\|x\|.
\end{align*}
Thus we obtain a contradiction. Therefore, we conclude that the inequality $\|x+\lambda y\|\ge \|x\|$ must hold
for every $\lambda\ge 0$.
\end{proof}

\section{Proper Acute and Proper Obtuse B-Angles}

If $X$ is equipped with an inner product $\langle\cdot,\cdot\rangle$ and the induced norm $\|x\|:=\langle x,x\rangle^{1/2}$,
we find that two nonzero vectors $x$ and $y$ form a proper acute angle if and only if there exists $\gamma>0$ such that
$\|x+\lambda y\|<\|x\|$ precisely for every $\lambda \in (-\gamma,0)$. Likewise, $x$ and $y$ form a proper obtuse angle
if and only if there exists $\gamma>0$ such that $\|x+\lambda y\|<\|x\|$ precisely for every $\lambda \in (0,\gamma)$.
The value of $\gamma$ in both cases is given by $\gamma:=\frac{2|\langle x,y\rangle|}{\|y\|^2}$. Our aim now
is to formulate the criteria for proper acute and proper obtuse B-angles in a normed space $(X,\|\cdot\|)$.

Let $x,y\in X$. We define that
\begin{itemize}
\item $x$ forms a {\bf proper acute B-angle} to $y$, denoted by $x\,PA_B\,y$, if $x\,A_B\,y$ but $x\,\not\perp_B\,y$.
\item $x$ forms a {\bf proper obtuse B-angle} to $y$, denoted by $x\,PO_B\,y$, if $x\,O_B\,y$ but $x\,\not\perp_B\,y$.
\end{itemize}
Note that, given $x,y\in X$, we now have three exclusive possibilities: $x\,PA_B\,y$, $x\,PO_B\,y$, or $x\,\perp_B\,y$.

As in inner product spaces, we obtain analogous results in normed spaces, as stated in the following theorem.

\begin{theorem}\label{T3-1}
Let $x,y\in X$. Then there are only three (exclusive) possibilities for the set of values of $\lambda$ for which
the inequality (\ref{B-ineq}) holds (or fails to hold), namely:
{\parindent=0cm
\begin{enumerate}
    \item There exists $\gamma>0$ such that the inequality (\ref{B-ineq}) fails to hold precisely for every $\lambda\in (-\gamma,0)$.
    \item There exists $\gamma>0$ such that the inequality (\ref{B-ineq}) fails to hold precisely for every $\lambda\in (0,\gamma)$.
    \item The inequality (\ref{B-ineq}) holds for every $\lambda\in \mathbb{R}$.
\end{enumerate}
}
\end{theorem}

To prove the theorem, we consider the following set
\[
S(x,y):=\{\lambda\in\mathbb{R} : \|x+\lambda y\|<\|x\|\},
\]
for $x,y\in X$. The above theorem is a consequence of the following statements.

\begin{theorem}\label{T3-2}
Let $x,y\in X$. Then each of the following statements hold:
{\parindent=0cm
\begin{enumerate}
    \item $0\notin S(x,y)$.
    \item If there exists $\lambda>0$ such that $\lambda \in S(x,y)$, then $(0,\lambda]\subseteq S(x,y)$ and $\mu\notin S(x,y)$ for every $\mu\le0$.
    \item If there exists $\lambda<0$ such that $\lambda \in S(x,y)$, then $[\lambda,0)\subseteq S(x,y)$ and $\mu\notin S(x,y)$ for every $\mu\ge0$.
    \item $S(x,y)$ is bounded.
    \item If $S(x,y)\not=\emptyset$, then $0=\inf S(x,y)<\sup S(x,y)$ {\bf or} $\inf S(x,y)<\sup S(x,y)=0$.
    \item If $S(x,y)\not=\emptyset$, then $\inf S(x,y)\notin S(x,y)$ {\bf and} $\sup S(x,y)\notin S(x,y)$.
\end{enumerate}
}
\end{theorem}

\begin{proof}
{\parindent=0cm
\begin{enumerate}
\item Obvious.
\item Let $\lambda>0$ and $\lambda\in S(x,y)$. It follows from the first statement that $0\notin S(x,y)$.
    Now, for $0<\lambda^\prime<\lambda$, we observe that
    \begin{align*}
        \left|\left|x+\lambda^\prime y\right|\right|&=\left|\left|\frac{\lambda-\lambda^\prime}{\lambda}x+\frac{\lambda^\prime}{\lambda}x+
        \lambda^\prime y\right|\right|
        \\ &\leq \frac{\lambda-\lambda^\prime}{\lambda}\|x\|+\frac{\lambda^\prime}{\lambda}\|x+\lambda y\|
        \\ &< \frac{\lambda-\lambda^\prime}{\lambda}\|x\|+\frac{\lambda^\prime}{\lambda}\|x\|
        \\ &=\|x\|,
    \end{align*}
    whence $\lambda^\prime\in S(x,y)$. This proves that $(0,\lambda]\subseteq S(x,y)$.
    Next, suppose that there exists $\mu<0$ such that $\mu\in S(x,y)$. Notice that for $t_0=\frac{-\mu}{\lambda-\mu}\in (0,1)$, we have
    \[0=(1-t_0)\mu+t_0 \lambda.\]
    Hence
    \begin{align*}
        \|x\|&=\|x+[(1-t_0)\mu+t_0\lambda]y\|
        \\ &\leq (1-t_0)\|x+\mu y\|+t_0\|x+\lambda y\|
        \\ &<(1-t_0)\|x\|+t_0\|x\|
        \\ &=\|x\|,
    \end{align*}
    which cannot be true. Therefore we conclude that $\mu\notin S(x,y)$ for every $\mu\le0$.
    \item Similar to the proof of (2).
    \item Assuming that $S(x,y)\neq \emptyset$, let $\lambda\in S(x,y)$. We observe that
    \[
    \big| \|x\|-|\lambda|\,\|y\|\big|\leq \|x+\lambda y\|<\|x\|.
    \]
    The assumption that $S(x,y)\neq \emptyset$ means that $x,y\not=0$. Using the above inequality, we obtain
    \[
    0<|\lambda|<\frac{2\|x\|}{\|y\|}.
    \]
    Hence $S(x,y)$ is bounded.
    \item It follows from (2), (3), and (4).
    \item Suppose that $S(x,y)\neq \emptyset$. Assume that there exists $\lambda^*>0$ such that $\lambda^*\in S(x,y)$.
    By (2), (4), and (5), $\inf S(x,y)=0\notin S(x,y)$ and $\sup S(x,y)$ exists. Let $s:=\sup S(x,y)$.
    Now the function $f:\mathbb{R}\to\mathbb{R}$ given by $f(\lambda)=\|x+\lambda y\|$ is continuous everywhere, with
    \begin{align*}
        f(\lambda)&< \|x\| \text{ if } \lambda\in (0,s),\ \text{and}
        \\ f(\lambda)&\geq \|x\| \text{ if } \lambda\leq 0\text{ or }\lambda>s.
    \end{align*}
    Hence
    \[
    \|x\|\leq \lim_{\lambda\to s^+}f(\lambda)=f(s)=\lim_{\lambda\to s^-}f(\lambda)\leq \|x\|.\]
    Thus $\|x+s y\|=f(s)=\|x\|$, which implies that $\sup S(x,y)\notin S(x,y)$. By similar arguments, if there
    exists $\lambda^*<0$ such that $\lambda^*\in S(x,y)$, then $\inf S(x,y)\notin S(x,y)$.
\end{enumerate}
}
\end{proof}

\begin{remark}
In view of Theorem \ref{T3-2}, the value of $\gamma$ in Theorem \ref{T3-1}, part (1), equals $-\inf S(x,y)$; while
that in part (2) equals $\sup S(x,y)$.
\end{remark}

From Theorem \ref{T3-1}, parts (1) and (2), we have the following corollary.

\begin{cor}
Let $x,y\in X$. Then we have
\begin{enumerate}
\item $x\,PA_B\,y$ if and only if there exists $\gamma>0$ such that (\ref{B-ineq}) fails to hold precisely for every $\lambda\in (-\gamma,0)$.
\item $x\,PO_B\,y$ if and only if there exists $\gamma>0$ such that (\ref{B-ineq}) fails to hold precisely for every $\lambda\in (0,\gamma)$.
\end{enumerate}
\end{cor}

As for acute and obtuse $B$-angles, we also have the following results for proper acute and proper obtuse B-angles.

\begin{prop}\label{invariant}
Let $x,y\in X\setminus \{0\}$ and $a,b\in\mathbb{R}\setminus\{0\}$.
{\parindent=0cm
\begin{enumerate}
    \item Suppose that $x\,PA_B\,y$. If $ab>0$, then $ax\,PA_B\,by$.
    If $ab<0$, then $ax\,PO_B\,by$.
    \item Suppose that $x\,PO_B\,y$. If $ab>0$, then $ax\,PO_B\,by$.
    If $ab<0$, then $ax\,PA_B\,by$.
\end{enumerate}
}
\end{prop}

\section{Comparing Two B-Angles}

Let $x,y\in X$. Throughout this section, we shall consider the case where $x\,PA_B\,y$. (The discussion
for the case where $x\,PO_B y$ is similar.) For the purpose of our discussion here, we write $\gamma(x,y)$
for the largest value of $\gamma$ for which the inequality (\ref{B-ineq}) fails to hold precisely for
every $\lambda\in (-\gamma,0)$.

Let $x,y_1,y_2\in X\setminus\{0\}$. If $x\,PA_B\,y_1$ and $x\,O_B\,y_2$ (which includes the possibility that
$x\perp_B y_2$), then it is safe to say that the B-angle from $x$ to $y_1$ is more acute than that from $x$
to $y_2$. The question now is: what can we say when $x\,PA_B\,y_1$ and $x\,PA_B\,y_2$? How can we tell that one
B-angle is more acute than the other? Suppose that $\gamma(x,y_1) < \gamma(x,y_2)$ (see the figure below).

\begin{center}
\begin{center}
\begin{tikzpicture}
               \draw [thick, black, -] (-3,0) -- (3,0) ;
               \draw [thick, blue, ->] (-1,0.25)--(-3,0.25) ;
               \draw [thick, blue, ->] (0,0.25)--(3,0.25) ;
               \draw [thick, black, -] (-1,-0.25) -- (-1,0.25) ;
               \draw [thick, black, -] (0,-0.25) -- (0,0.25) ;
               \fill (0,0) circle (0.05);
               \fill (-1,0) circle (0.05);
               \node [below] at (0,-0.25) {$0$};
               \node [below] at (-1,-0.25) {$-\gamma(x,y_1)$};
\end{tikzpicture}
\\
\begin{tikzpicture}
               \draw [thick, black, -] (-3,0) -- (3,0) ;
               \draw [thick, blue, ->] (-2,0.25)--(-3,0.25) ;
               \draw [thick, blue, ->] (0,0.25)--(3,0.25) ;
               \draw [thick, black, -] (-2,-0.25) -- (-2,0.25) ;
               \draw [thick, black, -] (0,-0.25) -- (0,0.25) ;
               \fill (0,0) circle (0.05);
               \fill (-2,0) circle (0.05);
               \node [below] at (0,-0.25) {$0$};
               \node [below] at (-2,-0.25) {$-\gamma(x,y_2)$};
\end{tikzpicture}
\end{center}
{\small
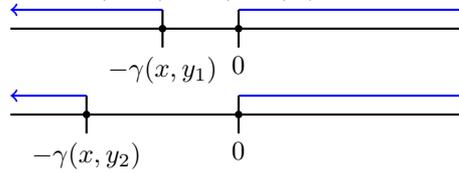
{\tt Figure 4.1}. The blue rays indicated the sets of the values of $\lambda$ for which $\|x+\lambda y_1\|\ge \|x\|$
and $\|x+\lambda y_2\|\ge \|x\|$ hold, respectively.
}
\end{center}

\medskip

Intuitively, we might want to conclude that the B-angle from $x$ to $y_2$ is more acute than that from $x$ to $y_1$
(that is, the larger the value of $\gamma(x,y)$, the more acute the B-angle from $x$ to $y$). However, this conclusion
is too early to make. Observe the following example in $\ell_2^\infty$.

Let $x:=(1,0),\ y_c:=(c,c) \in \ell_2^\infty$, where $c>0$. By observation, we see that $x\,PA_B\, y_c$ (see the figure below).

\begin{center}
\begin{center}
    \begin{tikzpicture}
               \draw [thick, gray, ->] (0,-2.5) -- (0,2.5)  
               node [above, black] {$x_2$};              

               \draw [thick, gray, ->] (-2.5,0) -- (2.5,0)      
               node [right, black] {$x_1$};              
               \draw [thick, green, -] (1,1)--(1,-1)--(-1,-1)--(-1,1)--(1,1)  ;    
               \draw [thick, black, ->] (0,0)--(1,0)  ;    
               \draw [thick,black , ->] (0,0)--(1.5,1.5)  ;    
               \draw [thick, gray, -] (-1.5,-2.5)--(2.5,1.5)  ;    
               \node [left] at (1.5,1.5) {$y_c$};
               \node [above] at (1,0) {$x$};
               \draw [thick, red, ->] (0,0)--(0.5,-0.5)  ;    
               \draw [thick, red, ->] (0,0)--(1.5,0.5)  ;    
               \draw [thick, red, ->] (0,0)--(2,1)  ;    
               \draw [thick, red, ->] (0,0)--(0,-1)  ;    
               \node [right,red] at (2,1) {$x+\lambda y_c$};
    \end{tikzpicture}
\end{center}
{\small
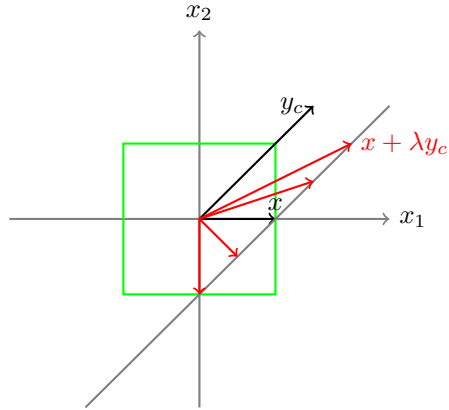
{\tt Figure 4.2}. The vector $x$ forms a proper acute B-angle to the vector $y$ in $\ell_2^\infty$.
}
\end{center}

\medskip

Let us now take two different values of $c$, say $c_1=1$ and $c_2=2$, so that we have $y_1:=(1,1)$
and $y_2:=(2,2)$. By simple calculation, we obtain $\gamma(x,y_1)=1$ and $\gamma(x,y_2)=\frac12$.
Thus $\gamma(x,y_2)<\gamma(x,y_1)$. However, we would {\bf not} say that the B-angle from $x$ to $y_1$ is
more acute than that from $x$ to $y_2$ since $y_2=2y_1$. In such a case, we would want to have the
B-angle from $x$ to $y_1$ {\bf equal to} that from $x$ to $y_2$. To compare two proper B-acute angles,
we need to find a number which in general depends on the two vectors but invariant under positive scalar
multiplications. In order to do so, we go back to inner product spaces, to get an inspiration.

Suppose, for the moment, $X$ is equipped with an inner product $\langle\cdot,\cdot\rangle$ and its
induced norm $\|x\|:=\langle x,x\rangle^{1/2}$. Let $x,y\in X\setminus\{0\}$, and consider the expression
$\frac{\langle x,y\rangle}{\|x\|\,\|y\|}$, which is nothing but the cosine of the angle between $x$ and
$y$. We note that the value of this expression, and so is the angle, is invariant under positive scalar
multiplications. Indeed, if $x^\prime:=ax$ and $y^\prime:=by$ with $a,b>0$, then
\[
\frac{\langle x',y'\rangle}{\|x'\|\,\|y'\|}=\frac{\langle ax,by\rangle}{\|ax\|\,\|by\|}=
\frac{ab}{ab}\frac{\langle x,y\rangle}{\|x\|\,\|y\|}=\frac{\langle x,y\rangle}{\|x\|\,\|y\|}.
\]
Thus, normalizing both vectors, we get
\[
\langle \hat x,\hat y\rangle=\frac{\langle x,y\rangle}{\|x\|\,\|y\|}
\]
where $\hat{x}:=\frac{x}{\|x\|}$ and $\hat{y}:=\frac{y}{\|y\|}$.

We turn back to our normed space $(X,\|\cdot\|)$. Let $x,y\in X\setminus\{0\}$. Our goal is to find
a number $\gamma^*=\gamma^*(x,y)$ such that $\gamma^*(ax,by)=\gamma^*(x,y)$ whenever $a,b>0$. The key is
the lemma below.

\begin{lemma}
Let $x,y\in X\setminus\{0\}$ and suppose that $x\,PA_B\,y$. If $a,b>0$, then
$ax\,PA_B\,by$ and
\[
\gamma(ax,by)=\frac{a}{b}\,\gamma(x,y).
\]
\end{lemma}

\begin{proof}
Let $a,b>0$. By Proposition \ref{invariant}, we have $ax\,PA_B\,by$. Now suppose that $\lambda\in(-\gamma(x,y),0)$,
where we have $\|x+\lambda y\|<\|x\|$. By simple manipulations, we find that
\[
\|ax+\frac{a}{b}\lambda (by)\|<\|ax\|.
\]
Thus $\frac{a}{b}\lambda \in (-\gamma(ax,by),0)$. Since we also have $\frac{a}{b}\lambda \in (-\frac{a}{b}\gamma(x,y),0)$,
we conclude that $\gamma(ax,by)\ge \frac{a}{b}\gamma(x,y)$. On the other hand, since $x=\frac{1}{a}(ax)$ and $y=\frac{1}{b}(by)$,
we obtain $\gamma(ax,by)\le \frac{a}{b}\gamma(x,y)$ via similar arguments. As a consequence, we arrive at the conclusion that
\[
\gamma(ax,by)=\frac{a}{b}\gamma(x,y),
\]
as claimed.
\end{proof}

We are now ready to define the number $\gamma^*$ with the desired property. Let $x,y\in X$ such that $x\,PA_B\,y$. Here
of course $x,y\not=0$, for otherwise $x$ will be B-orthogonal to $y$ per definition. We then define
\begin{equation}\label{gamma*}
\gamma^*(x,y):=\frac{\|y\|}{\|x\|}\gamma(x,y).
\end{equation}
Accordingly, we have the following lemma.

\begin{lemma}
Let $x,y\in X\setminus\{0\}$ and suppose that $x\,PA_B\,y$. If $a,b>0$, then $\gamma^*(ax,by)=\gamma^*(x,y)$.
\end{lemma}

\begin{proof}
Let $a,b>0$. By (\ref{gamma*}), we have
\[
\gamma^*(ax,by)=\frac{\|by\|}{\|ax\|}\gamma(ax,by)=\frac{b\|y\|}{a\|x\|}\frac{a}{b}\gamma(x,y)=\frac{\|y\|}{\|x\|}\gamma(x,y)=\gamma^*(x,y),
\]
as desired.
\end{proof}

Based on the above lemma, given $x,y\in X$ such that $x\,PA_B\,y$, we have $\gamma^*(x^\prime,y^\prime)=\gamma^*(x,y)$ for every
$x^\prime=ax$ and $y^\prime=by$ with $a,b>0$. Using this fact, we can now consider the unit vectors $\hat{x}:=\frac{x}{\|x\|}$ and
$\hat{y}:=\frac{y}{\|y\|}$, where we have
\[
\gamma^*(x,y)=\gamma^*(\hat x,\hat y)=\gamma(\hat x,\hat y).
\]
This equality tells us that to compare acute B-angles, it suffices for us to compare the values of $\gamma(\cdot,\cdot)$
among the associated unit vectors. The same is also true for $x,y\in X$ for which $x\,PO_B\,y$. To be precise, we define the following.

Let $x,y_1,y_2 \in X\setminus\{0\}$ such that $x\,PA_B\,y_1$ and $x\,PA_B\,y_2$. Let $\hat{x},\ \hat{y}_1,$ and $\hat{y}_2$ be the unit vectors
associated to $x,y_1,$ and $y_2$, respectively. We say that
\begin{itemize}
\item the B-angle from $x$ to $y_1$ is {\bf more acute} than that from $x$ to $y_2$ if $\gamma(\hat x,\hat{y}_1)>\gamma(\hat x,\hat{y}_2)$.
\item the B-angle from $x$ to $y_1$ is {\bf the same} as that from $x$ to $y_2$ if $\gamma(\hat x,\hat{y}_1)=\gamma(\hat x,\hat{y}_2)$.
\item the B-angle from $x$ to $y_1$ is {\bf more obtuse} than that from $x$ to $y_2$ if $\gamma(\hat x,\hat{y}_1)<\gamma(\hat x,\hat{y}_2)$.
\end{itemize}
(Remark that similar definitions can be formulated for $x,y_1,y_2 \in X\setminus\{0\}$ such that $x\,PO_B\,y_1$ and $x\,PO_B\,y_2$.)

In the same spirit, we can also compare the B-angles for $x_1,x_2,y\in X\setminus\{0\}$ such that $x_1\,PA_B\,y$ and $x_2\,PA_B\,y$. Here we say
that
\begin{itemize}
\item the B-angle from $x_1$ to $y$ is {\bf more acute} than that from $x_2$ to $y$ if $\gamma(\hat{x}_1,\hat{y})>\gamma(\hat{x}_2,\hat{y})$.
\item the B-angle from $x_1$ to $y$ is {\bf the same} as that from $x_2$ to $y$ if $\gamma(\hat{x}_1,\hat{y})=\gamma(\hat{x}_2,\hat{y})$.
\item the B-angle from $x_1$ to $y$ is {\bf more obtuse} than that from $x_2$ to $y$ if $\gamma(\hat{x}_1,\hat{y})<\gamma(\hat{x}_2,\hat{y})$.
\end{itemize}
(As in the previous case, similar definitions can be formulated for $x_1,x_2,y \in X\setminus\{0\}$ such that $x_1\,PO_B\,y$ and $x_2\,PO_B\,y$.)

\begin{example}\label{Example1}
{\rm
Let $x:=(1,0),y:=(y_1,y_2)\in \ell_2^\infty$. Firstly, let us consider the case where $y:=(a,1)$ with $0<a<1$.

\begin{center}
\begin{center}
\begin{tikzpicture}
               \draw [thick, gray, ->] (0,-3) -- (0,3)  
               node [above, black] {$x_2$};              
               \draw [thick, gray, ->] (-3,0) -- (3,0)      
               node [right, black] {$x_1$};              
               \draw [thick, green, -] (1,1)--(1,-1)--(-1,-1)--(-1,1)--(1,1)  ;    
               \draw [thick, black, ->] (0,0) -- (1,0)      
               node [above, black] {$x=\hat x$};              
               \draw [thick, red, ->] (0,0) -- (0.5,1)      
               node [above, red] {$y=\hat y$};              
\end{tikzpicture}
\end{center}
{\small
{\tt Figure 4.3}. The vectors $x=(1,0)$ and $y=(a,1)$ with $0<a<1$ in $\ell_2^\infty$.
}
\end{center}

\medskip

Note that $\hat x =x$ and $\hat y =y$, and $x\,PA_B\,y$. Indeed, there exists $\gamma(x,y)>0$ such that $\|x+\lambda y\|<\|x\|=1$
precisely for every $\lambda\in (-\gamma(x,y),0)$. To find the value of $\gamma(x,y)$, we solve the inequality
\[
\max\{|1+\lambda a|,|\lambda|\}=\|x+\lambda y\|<\|x\|=1
\]
for $\lambda$. We obtain that $\lambda$ must satisfy $-2/a<\lambda<0$ and $-1<\lambda<0$. Since $0<a<1$, we have $1/a>1$, and so
we find that $\lambda\in (-1,0)$. Therefore $\gamma(x,y)=1$, independent of the value of $a\in(0,1)$. In this case, the B-angle
from $x=(1,0)$ to $y_1=(a_1,1)$ is the same as that from $x=(1,0)$ to $y_2=(a_2,1)$ whenever $0<a_1,a_2<1$.

Secondly, let us consider the case where $x:=(1,0),y:=(1,a)$ with $0<a\leq1$.

\begin{center}
\begin{center}
\begin{tikzpicture}
               \draw [thick, gray, ->] (0,-3) -- (0,3)  
               node [above, black] {$x_2$};              
               \draw [thick, gray, ->] (-3,0) -- (3,0)      
               node [right, black] {$x_1$};              
               \draw [thick, green, -] (1,1)--(1,-1)--(-1,-1)--(-1,1)--(1,1)  ;    
               \draw [thick, black, ->] (0,0) -- (1,0)      
               node [above, black] {$x=\hat x$};              
               \draw [thick, red, ->] (0,0) -- (1,0.5)      
               node [above, red] {$y=\hat y$};              
\end{tikzpicture}
\end{center}
{\small
{\tt Figure 4.4}. The vectors $x=(1,0)$ and $y=(1,a)$ with $0<a\leq 1$ in $\ell_2^\infty$.
}
\end{center}

\medskip

Note that $\hat y =y$ and $x\,PA_B\,y$. The value of $\gamma(x,y)>0$ for which $\|x+\lambda y\|<\|x\|=1$ precisely for every
$\lambda\in (-\gamma(x,y),0)$ can be found by solving the inequality
\[
\max\{|1+\lambda|,|a\lambda|\}=\|x+\lambda y\|<\|x\|=1
\]
for $\lambda$. Here we obtain that $-2<\lambda<0$ and $-\frac{1}{a}<\lambda<\frac{1}{a}$, which yields $\lambda\in (-\min\{2,\frac1{a}\},0)$.
For $0<a\leq\frac12$, we have $\frac1{a}\ge2$, and so $\min\{2,\frac1{a}\}=2$. For $\frac12<a\leq 1$, we have $1\leq \frac1{a}<2$,
which gives $\min\{2,\frac1{a}\}=\frac1{a}$.
Therefore $\gamma(x,y)=2$ for $a\in (0,1/2]$ and $\gamma(x,y)=\frac{1}{a}$ for $a\in(\frac12,1]$.
In this example, the B-angle from $x=(1,0)$ to $y_1=(1,a_1)$ is more acute than that from $x=(1,0)$ to $y_2=(1,a_2)$ provided that
$0<a_1\le \frac12 <a_2\le1$.
}
\end{example}

\begin{example}\label{Example2}
{\rm
Let $x:=(1,0),y=(a,\sqrt{1-a^2})\in \ell_2^2$ with $0<a<1$.
\begin{center}
    \begin{center}
    \begin{tikzpicture}
               \draw [thick, gray, ->] (0,-3) -- (0,3)  
               node [above, black] {$x_2$};              

               \draw [thick, gray, ->] (-3,0) -- (3,0)      
               node [right, black] {$x_1$};              
               \draw [green,line width=1pt] (0,0) circle (1cm) ;
               \draw [cyan,line width=1pt] (1,0) circle (1cm) ;
               \draw [thick, black, ->] (0,0) -- (0.7071,0.7071)      
               node [below, black] {$y$};
               \draw [thick, black, ->] (0,0) -- (1,0)      
               node [below, black] {$x$};              
               \draw [thick, red, ->] (0,0) -- (1.7071,0.7071)      
               node [right, red] {$x+y$};              
               \draw [thick, gray, -] (-2,-3) -- (3,2)  ;    
               \node [left,gray] at (2.5,1.5) {$x+\lambda y$};
\end{tikzpicture}
\end{center}
{\small
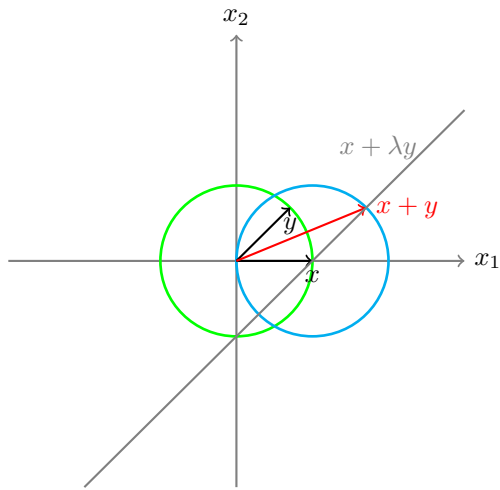
{\tt Figure 4.5}. The vectors $x=(1,0)$ and $y=(a,\sqrt{1-a^2})$ with $0<a<1$ in $\ell_2^2$.
}
\end{center}

\medskip

Observe that $\|x\|=\|y\|=1$ and $x\,PA_B\,y$. We can find the value of $\gamma(x,y)>0$ for which $\|x+\lambda y\|<\|x\|=1$ precisely
for every $\lambda\in (-\gamma(x,y),0)$. By simple computations, one may obtain that
\[
\gamma(x,y)=\frac{2\langle x,y\rangle}{\|y\|^2}=\frac{2a}{a^2+(1-a^2)}=2a.
\]
Thus here $\gamma(x,y)$ gets larger as $a$ tends to 1, that is, as $y$ is approaching $x$.
}
\end{example}

\section{Concluding Remarks: The Analog of the Cosine of B-Angles}

We have defined acute and obtuse B-angles in a normed space $(X,\|\cdot\|)$ via the sets of values of $\lambda$ for which
the inequality (\ref{B-ineq}) holds (or fails to hold), and defined the criteria for proper acute and proper
obtuse B-angles from a vector $x$ to another vector $y$ in $X$. We have also defined a number $\gamma^*(\cdot,\cdot)$
which can be used to compare the B-angle from $x$ to $y_1$ and that from $x$ to $y_2$, or from $x_1$ to $y$ and
that from $x_2$ to $y$. It is then tempting to have a formula for the B-angle from a vector to another vector in $X$, as
in an inner product space.

As the readers might have guessed by now, the formula is only one step away from the formula of $\gamma(\cdot,\cdot)$ on
the unit sphere. Let $x,y\in X\setminus\{0\}$. Define
\[
k(x,y)=\begin{cases}\frac{1}{2}\gamma(\hat{x},\hat{y}), &\text{if } x\,PA_B\,y;
\\ 0, &\text{if } x\perp_B y;
\\ -\frac{1}{2}\gamma(\hat{x},\hat{y}), &\text{if } x\,PO_B\,y.
\end{cases}
\]
Notice that if $X$ is an inner product space which is also equipped with the induced norm, then
$k(x,y)=\langle \hat{x},\hat{y}\rangle$, which is nothing but the cosine of the angle between $x$ and $y$.

The following proposition gives basic properties of the function $k(x,y)$, which may be viewed as
the {\bf cosine of the B-angle} from $x$ to $y$.

\begin{prop}
Let $x,y\in X\setminus\{0\}$ and $a,b\not=0$. Then the following statements hold:
\begin{enumerate}
    \item $|k(x,y)|\leq 1$.
    \item If $ab>0$, then $k(ax,by)=k(x,y)$. If $ab<0$, then $k(ax,by)=-k(x,y)$.
\end{enumerate}
\end{prop}

\begin{proof}
\begin{enumerate}
    \item Suppose that we have some values of $\lambda$ satisfying $\|\hat{x}+\lambda \hat{y}\|<1$. Then
    \[
    \big| \|\hat{x}\|-|\lambda|\,\|\hat{y}\|\big|\leq \|\hat{x}+\lambda \hat{y}\|<1.
    \]
    Hence we have $\big|1-|\lambda|\big|<1$, whence $0<|\lambda|<2$. This implies that $\gamma(\hat{x},\hat{y})\leq 2$,
    and therefore $|k(x,y)|= \frac{1}{2}\gamma(\hat{x},\hat{y})\leq 1$.
    \item We shall only prove it for the case where $ab<0$ and $x\,PA_B\,y$, and leave the other cases to the readers.
    Without loss of generality, assume that $a>0>b$. The hypothesis that $x\,PA_B\,y$ means that there exists $\gamma(\hat{x},\hat{y})>0$
    such that $\|\hat {x}+\lambda \hat{y}\|<1$ precisely for every $\lambda \in (-\gamma(\hat{x},\hat{y}),0)$.
    Since $a>0>b$, we have $ax\,PO_B\,by$, which means that there exists $\gamma(\widehat{ax},\widehat{by})>0$ such that
    $\|\widehat {ax}+\lambda \widehat{by}\|<1$ precisely for every $\lambda \in (0,\gamma(\widehat{ax},\widehat{by}))$.
    Now the assumption that $a>0>b$ implies that $\widehat{ax}=\hat{x}$ and $\widehat{by}=-\hat{y}$. For
    $0<\lambda^\prime<\gamma(\hat{x},\hat{y})$, set $\lambda=-\lambda^\prime$. Then $-\gamma(\hat{x},\hat{y})<\lambda<0$, 
    and so we obtain
    \[
    \|\widehat{ax}+\lambda^\prime(\widehat{by})\|=\|\hat{x}+\lambda^\prime(-\hat{y})\|=\|\hat{x}+\lambda\hat{y}\|<1.
    \] 
    This implies that $\gamma(\hat{x},\hat{y})\leq \gamma(\widehat{ax}, \widehat{by})$. 
    Conversely, by setting $x^\prime=ax, y^\prime=by$, we have $x=\frac1{a}x^\prime, y=\frac1{b}y^\prime$, 
    so that $\gamma(\hat{x},\hat{y})\geq \gamma(\widehat{ax}, \widehat{by})$.
    Therefore $\gamma(\widehat{ax}, \widehat{by})=\gamma(\hat{x},\hat{y})$. Consequently, we obtain
    \[
    k(x,y)=\frac{1}{2}\gamma(\hat{x},\hat{y})=\frac{1}{2}\gamma(\widehat{ax}, \widehat{by})=-k(ax,by),
    \]
    as expected.
\end{enumerate}
\end{proof}

Both properties in the above proposition are basic properties of the cosine of the B-angles in $X$.
We miss, however, the symmetric property $k(x,y)=k(y,x)$. To see that we do not have this property,
take for an example $x:=(1,0)$ and $y=(1,1)$ in $\ell_2^\infty$. Here $x\,PA_B\,y$ and $y\perp_B x$,
and thus $k(x,y)>0=k(y,x)$. This example tells us that, in general, $k(x,y)\not=k(y,x)$. (This is
why we do not use the phrase the B-angle {\bf between} $x$ {\bf and} $y$, but the B-angle {\bf from}
$x$ {\bf to} $y$.)

Furthermore, given a vector $x\in X$, the set of vectors $y$ such that $k(x,y)=1$ may consists more than
just $y=x$. We have seen this in Example \ref{Example1} where $X=\ell_2^\infty,\ x:=(1,0)$ and $y:=(1,a)$
with $a\in[0,1/2)$. The same also happens with the set of vectors $y$ such that $k(x,y)=-1$.

Similar to Example \ref{Example1}, we can also compute the values of $k(x,y)$ for $x:=(1,0)$ and
$y:=(\cos \theta,\sin \theta) \in \ell_2^1$ with $\theta\in (-\pi,\pi]$. Here $k$ can be seen as a
function of $\theta$, say $k=f(\theta)$, where
\[
f(\theta):=\begin{cases}1, &-\frac{\pi}{4}<\theta<\frac{\pi}{4};
\\0,&-\frac{3\pi}{4}\leq \theta\leq-\frac{\pi}{4}\text{ or }\frac{\pi}{4}\leq \theta\leq\frac{3\pi}{4};
\\-1,&-\pi<\theta<-\frac{3\pi}{4}\text{ or }\frac{3\pi}{4}< \theta\leq\pi.
\end{cases}
\]
The graph of $k=f(\theta)$ is presented below:
\begin{center}
    \begin{center}
    \begin{tikzpicture}
               \draw [thick, black, ->] (0,-2) -- (0,2)  
               node [above, black] {$f(\theta)$};              

               \draw [thick, gray, ->] (-4,0) -- (4,0)      
               node [right, black] {$\theta$};              
               \draw [thick, blue,-] (-3.1415,-1) -- (-2.3561,-1);
               \draw [thick, blue, -] (-2.3561,0) -- (-0.7853,0);
               \draw [thick, blue, -] (-2.3561,0) -- (-0.7853,0);
               \draw [thick, blue, -] (-0.7853,1) -- (0.7853,1);
               \draw [thick, blue, -] (0.7853,0) -- (2.3561,0);
               \draw [thick, blue, -] (2.3561,-1) -- (3.1415,-1);
               \draw [blue, dashed, -] (-3.1415,0) -- (-3.1415,-1);
               \draw [blue, dashed, -] (-2.3561,-1) -- (-2.3561,0);
               \draw [blue, dashed, -] (-0.7853,0) -- (-0.7853,1);
               \draw [blue, dashed, -] (0.7853,0) -- (0.7853,1);
               \draw [blue, dashed, -] (2.3561,-1) -- (2.3561,0);
               \draw [blue, dashed, -] (3.1415,0) -- (3.1415,-1);
               \draw[blue,thick] (-2.3561,-1) circle (0.075 cm);
               \draw[fill,blue] (-2.3561,0) circle (0.075 cm);
               \draw[fill,blue] (-0.7853,0) circle (0.075 cm);
               \draw[blue,thick] (-0.7853,1) circle (0.075 cm);
               \draw[blue,thick] (0.7853,1) circle (0.075 cm);
               \draw[fill,blue] (0.7853,0) circle (0.075 cm);
               \draw[fill,blue] (2.3561,0) circle (0.075 cm);
               \draw[blue,thick] (2.3561,-1) circle (0.075 cm);
               \draw[fill,blue] (3.1415,-1) circle (0.075 cm);
               \node [below] at (-3.1415,0) {$-\pi$};
               \node [below] at (-2.3561,0) {$-\frac{3\pi}{4}$};
               \node [below] at (-0.7853,0) {$-\frac{\pi}{4}$};
               \node [below] at (0.7853,0) {$\frac{\pi}{4}$};
               \node [below] at (2.3561,0) {$\frac{3\pi}{4}$};
               \node [below] at (3.1415,0) {$\pi$};
\end{tikzpicture}
\end{center}
{\small
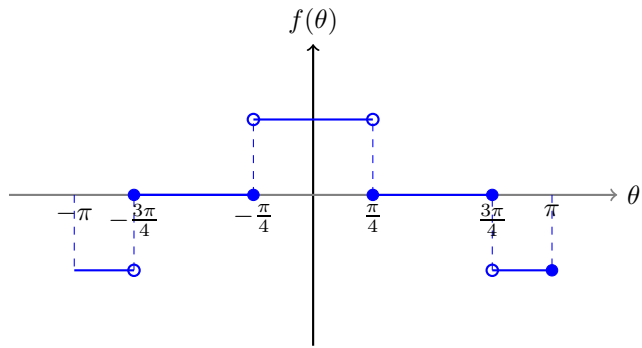
{\tt Figure 5.1}. The graph of $k$ as a function of $\theta$ for $\theta\in(-\pi,\pi]$.
}
\end{center}

\bigskip

\noindent{\bf Acknowledgements}. This work is supported by P2MI-ITB 2021 Program.

\end{document}